\documentclass[11pt]{amsart}        % Your input file must contain these two
\usepackage[]{Mytheorems}
\usepackage{amssymb,amscd,amsthm,amsxtra}
\usepackage{psfrag}
\usepackage{mathrsfs,amscd}

\usepackage{latexsym}
\input epsf

\usepackage{color}

\begin{document}
\title{ Singularity of  macroscopic variables near boundary for gases  with  cut-off inverse-power potential      }
\author[IC]{I-Kun Chen}
\address{Department of Mathematics, National Taiwan University, Taipei 10617, Taiwan }
\email{ikun.chen@gmail.com}
\author[CH]{Chun-Hsiung Hsia}
\address{Institute of Applied Mathematical Sciences, National Taiwan University, Taipei 10617, Taiwan }
\email{willhsia@math.ntu.edu.tw}
\date{\today}
\begin{abstract}
In this article, the boundary singularity for stationary solutions of the linearized Boltzmann equation with cut-off inverse power potential is analyzed. In particular, for cut-off hard-potential cases, we establish the asymptotic approximation for the gradient of the moments. Our analysis indicates the logarithmic singularity of the gradient of the moments.
\end{abstract}
 \maketitle
 \section{introduction}
 The Boltzmann equation is 
 \begin{equation}\frac{\partial F}{\partial t}+\xi \cdot\frac{\partial F}{\partial t}=Q(F,F),\end{equation}
 where $F=F(t,x,\xi)$.   $Q$ above  is the collision operator only involves velocity as follows:
 \begin{align}\label{QFF} Q(F,F)&=\int \int_0^{2\pi}\int_0^{\frac{\pi}2}(F'F_*'-FF_*)B(|V|,\theta) d\theta d \epsilon d\xi_*,\end{align}
 where $V=\xi_*-\xi$ and $\alpha$ is a unit vector on a hemisphere  parametrized by $\theta$ and $\epsilon$ such that $\alpha\cdot V=|V|\cos\theta $ and
 \begin{align}
 &F=F(\xi),\ \ F_*=F(\xi_*),\ F'=F(\xi'),\ \ F_*'=F(\xi_*'),\\&
 \xi'=\xi+(\alpha\cdot V)\alpha,\\&
 \xi_*'=\xi_*-(\alpha\cdot V)\alpha.
 \end{align}
 The $B(|V|,\theta)\geq 0$ is called the cross-section. If we consider inverse power force between particles, i.e.,  $Force= \frac1{r^s}$,
 then the  cross-section is in the form
 \begin{equation}
B(|V|,\theta)=|V|^{\gamma}\beta(\theta),
 \end{equation}
 where  $\gamma=\frac{s-5}{s-1}$. The fact  $\beta\sim (\frac{\pi}2-\theta)^{-\frac{s+1}{s-1}}$ as $\theta\to \frac{\pi}2$, which is not integrable in $\theta$, makes us unable to separate \ref{QFF} into gain and lost parts.  To avoid this mathematical difficulty, it was Grad's idea,  \cite{Grad}, to consider the cross-section such that 
 \begin{equation}
 B(|V|,\theta)\leq C|V|^{\gamma}\cos\theta\sin\theta.
 \end{equation}
 We will refer these cases as Grad's angular cut-off potential. 
 In particular, in our research we will consider the cases that $B$ as a product of a function of $|V| $ and one of $\theta$, i.e.,
  \begin{equation}\label{IP}\begin{split}
 B(|V|,\theta)=|V|^{\gamma}\beta(\theta),\ \ 
\beta(\theta)\leq C\cos\theta\sin\theta.\end{split}
 \end{equation}
 To make distinctions, we will refer the cases above, \eqref{IP}, as  power-law potential with angular cut-off in this paper. 
 We first non-dimensialize the equation so that the Maxwellian  we are interested becomes the standard one:
 \begin{align}w=\frac{1}{\sqrt{\pi}^3}e^{-{|\zeta|^2}}.\end{align}

  We  linearize the  equation around standard Maxwellian  so that
  \begin{align}F=w+w^{\frac12}f.\end{align}
  We have,

  \begin{equation}sh\frac{\partial f}{\partial t}+\zeta \cdot\frac{\partial f}{\partial t}=\frac1{\kappa}L(f),\end{equation}
  where $L(f)=2w^{-\frac12}Q(w,w^\frac12f).$ Under the assumption of Grad's angular cut-off, the linearized collision operator can be decomposed into a damping multiplicative operator $-\nu$ and a smoothing integral operator $K$: \begin{equation}L(\phi)(\zeta)=-\nu(\zeta)\phi(\zeta)+K(\phi)(\zeta).\label{decomposeL}\end{equation} 
The the following properties of the linearized collision operator were studied by Grad \cite{Grad} and Caflisch \cite{Caflisch}.   The collision frequency satisfies the following estimate
  \begin{equation}\nu_0(1+|\zeta|)^\gamma\leq \nu(\zeta)\leq \nu_1(1+|
  \zeta|)^\gamma,
  \end{equation}
  where $0<\nu_0<\nu_1$ and $-2\leq\gamma\leq1$ is a parameter from interaction between particles.  $\gamma=1$ is called the hard sphere model; $\gamma=0$ is called the Maxwellian case. Positive $\gamma$'s corresponded to hard potential; negative $\gamma$'s correspond to soft-potential. 
  
If we  consider power-law potential with angular cut-off, we have further properties :\begin{equation}\nu(\zeta)\mbox{ is a function of } |\zeta|.
 \end{equation}
 \begin{equation} \label{dKLp}\Vert \partial_{\zeta_i} K(f)\Vert_{L^p}\leq C\Vert f\Vert_{L^p}, \ p\geq 1.\end{equation} 
  
  In this paper, we restrict our study to the cases $0<\gamma\leq1$.  We define 
   \begin{equation}\Vert f(\zeta)\Vert_{L^\infty_a}=\sup_{\zeta}(1+|\zeta|)^a|f(\zeta)|.\end{equation} 
   
   The integral operator  improves the decay:
   \begin{align}
&\label{L2Linfty}\Vert K(f)\Vert_{L^\infty_{\frac32-\gamma}} \leq C\Vert f\Vert_{L^2}, \\ &\Vert K(f)\Vert_{L^\infty_{2+a-\gamma}} \leq C\Vert f\Vert_{L_a^\infty}.\label{Linfty}\end{align}

  We consider the stationary equation:
   
    \begin{equation}
 \zeta_1\partial_xf(x,\zeta)=L(f)(\zeta),  \mbox{for}  \  0<x\leq l. \label{localf}
  \end{equation}  
  The functional space we are considering is as follows:
  
  \begin{definition}\begin{align} L^*_\zeta(\mathbb R^3) &=\{ f,  \Vert f(\zeta)\Vert_*<\infty\},
  \end{align}where \begin{equation}  \Vert f(\zeta)\Vert_*=\Big (\int f^2(\zeta)\nu(\zeta)d\zeta\Big)^{\frac12}.\end{equation}
  Also, \begin{equation}|||f|||:=\sup_{0\leq x\leq l}||f||_*.\end{equation}
  \end{definition}
 We say $f\in L^\infty_x\big([0,l],L_\zeta^*(\mathbb{R}^3)\big)$ is a solution to \eqref{localf} if it satisfies the following integral equation:
 
  \begin{equation}\label{Integraleq1}
  f(x,\zeta)=\left\{\begin{array}{ll}e^{-\frac{\nu(\zeta)}{|\zeta_1|}x}f(0,\zeta)+\int_0^x\frac1{|\zeta_1|}e^{-\frac{\nu(\zeta)}{|\zeta_1|}(x-s)}K(f)(s,\zeta)ds, & \mbox{for}\  \zeta_1>0,\\ e^{-\frac{\nu(\zeta)}{|\zeta_1|}(l-x)}f(l,\zeta)+\int_x^l\frac1{|\zeta_1|}e^{-\frac{\nu(\zeta)}{|\zeta_1|}(s-x)}K(f)(s,\zeta)ds, & \mbox{for} \ \zeta_1<0.\end{array}\right.\end{equation}
 \begin{remark} 
 The solution spaces for both Milne and Kramar's problems given in \cite{GP} are in $ L^\infty_x\big([0,l],L_\zeta^*(\mathbb{R}^3)\big)$ if  $x$ is restricted  to $[0,l].$
 \end{remark} 
 
     The moments are defined as follows:
  \begin{definition}\label{moments}The $\alpha$ moment is defined as
  \begin{equation}
  \sigma_\alpha(x)=\int f(x,\zeta)\phi_\alpha (\zeta)d\zeta,
  \end{equation}
  where \begin{equation}\notag
\alpha=(\alpha_1,\alpha_2,\alpha_3),    \mbox{ $\alpha_i'$s are nonnegative  integers,  }\end{equation}and\begin{equation}
\phi_\alpha(\zeta)=\zeta^\alpha E^{\frac12}=\pi^{-\frac34}\zeta_1^{\alpha_1}\zeta_2^{\alpha_2}\zeta_3^{\alpha_3}e^{-\frac{|\zeta|^2}2}.\end{equation}
  \end{definition} 
  We introduce a constant depending on $\alpha$:
  \begin{definition} Set
  \begin{align}A_\alpha&=(2\alpha_1)^{\frac{\alpha_1}2}(2\alpha_2)^{\frac{\alpha_2}2}(2\alpha_3)^{\frac{\alpha_3}2}e^{-\frac{\alpha_1+\alpha_2+\alpha_3}2}, 
  \end{align}where we follow the convention $0^0=1$.
  \label{defaalpha}
  \end{definition}
   The macroscopic variables are obtained through the moments. For example, $\sigma_{(0,0,0)}$ is the density, $\sigma_{(1,0,0)}$ is the  flow velocity in $x_1$ direction, and $\frac23(\sigma_{(2,0,0)}+\sigma_{(0,2,0)}+\sigma_{(0,0,2)})-\sigma_{(0,0,0)}$ is the temperature.
   The following inequality will be frequently used later :\begin{equation}|\phi_\alpha|\leq C A_\alpha e^{-\frac{|\zeta|^2}{4}}.\end{equation}
  
The Main Theorem is as follows \begin{theorem}\label{mainthm}Suppose $f\in L^\infty_x\big([0,l],L_\zeta^*(\mathbb{R}^3)\big)$ is a solution to \eqref{localf}  with power law potential with angular cut-off with $0<\gamma\leq1$ and
  $  \nabla f(0,\zeta)\in L^p_\zeta(\mathbb{R}^{3+})$ for $p>1$, $f(0,\zeta)\in  L^ \infty_\zeta(\mathbb{R}^{3+})$ , and $f(l,\zeta)\in  L^ \infty_\zeta(\mathbb{R}^{3-}) $.  Then,  for $x$ small enough, 
 \begin{equation}\begin{split}|\partial_x\sigma_\alpha(x)|= &-\ln x\int\int \phi_\alpha(0,\zeta_2,\zeta_3)L(f)(0,0^+,\zeta_2,\zeta_3)d\zeta_2d\zeta_3\\&+O(A_\alpha \langle f \rangle'),\end{split}
\end{equation}  where
\begin{equation}
L(f)(0,0^+,\zeta_2,\zeta_3):=\lim_{\zeta_1\to 0^+}L(f)(0,\zeta_1,\zeta_2,\zeta_3)\end{equation}  and
\begin{equation}\langle f\rangle':=1+|||f|||+\Vert f(0,\zeta)\Vert_{L^ \infty_\zeta(\mathbb{R}^{3+})}+\Vert f(l,\zeta)\Vert_{L^ \infty_\zeta(\mathbb{R}^{3-})}+\Vert \nabla f(0,\zeta)\Vert_{L^ p_\zeta(\mathbb{R}^{3+})}.\end{equation}\end{theorem}

We first investigate the problem for Grad's angular cut-off potential.  We consider the the distribution function for $\zeta_1>0$. The case for $\zeta_1<0$ can be treated  similarly.  Differentiating \eqref{Integraleq1} for $\zeta_1>0$, we have

\begin{equation}\begin{split}\label{Dsigmap}
\frac{\partial}{\partial x}f(x,\zeta)=&-\frac{\nu(\zeta)}{|\zeta_1|}e^{-\frac{\nu(\zeta)}{|\zeta_1|}x}f(0,\zeta)+\frac{1}{|\zeta_1|}K(f)(x,\zeta)\\&-\int_0^x\frac{\nu(\zeta)}{|\zeta_1|^2}e^{-\frac{\nu(\zeta)}{|\zeta_1|}(x-s)}K(f)(s,\zeta)ds.\end{split} \end{equation} 
We  observe that the first  term is nice in $\zeta$ when $x$ is away from zero and has a singularity at  $x=0$. On the other hand, the second and third terms have  factors $|\zeta_1|^{-1}$ and $|\zeta_1|^{-2}$ in $\zeta$, which cause a  difficulty in our analysis.  In order to overcome this difficulty, we reorganize  the equation \eqref{Dsigmap}:
\begin{equation}\label{dfdxf}\begin{split}
\frac{\partial}{\partial x}f(x,\zeta)=&-\frac{\nu(\zeta)}{|\zeta_1|}e^{-\frac{\nu(\zeta)}{|\zeta_1|}x}f(0,\zeta)+\frac{1}{|\zeta_1|}e^{-\frac{\nu(\zeta)}{|\zeta_1|}x}K(f)(x,\zeta)\\&+\int_0^x\frac{\nu(\zeta)}{|\zeta_1|^2}e^{-\frac{\nu(\zeta)}{|\zeta_1|}(x-s)}\big(K(f)(x,\zeta)-K(f)(s,\zeta)\big)ds\\&{=\frac{1}{|\zeta_1|}e^{-\frac{\nu(\zeta)}{|\zeta_1|}x}L(f)(0,\zeta)+\frac{1}{|\zeta_1|}e^{-\frac{\nu(\zeta)}{|\zeta_1|}x}(K(f)(x,\zeta)-K(f)(0,\zeta))}\\&{+\int_0^x\frac{\nu(\zeta)}{|\zeta_1|^2}e^{-\frac{\nu(\zeta)}{|\zeta_1|}(x-s)}\big(K(f)(x,\zeta)-K(f)(s,\zeta)\big)ds.}\end{split}
\end{equation} 
The contribution from the second and third terms is uniformly bounded. Roughly speaking,  we organize the equation in such a way in order to use the "H{\"o}lder continuity" of $K(f)$ in $x$ to obtain  "differentiability" of $f$  in a very  week sense, which will be explained in detail in Section \ref{section2}. Also in Section \ref{section2}, we will deal with the contribution from $\zeta_1<0$.  In section \ref{section3}, with further assumption of invers-power potential with angular cut-off and regularity on boundary data, we can extract the 
singularity from the contribution of the first term on the right hand side of  \eqref{dfdxf}, which concludes the Theorem \ref{mainthm}.

\section{Upper Bound Estimates \label{section2} }
As outlined in the introduction,  the goal of this section is to prove \begin{lemma}\label{lmGrad}Suppose $f\in L^\infty_x\big([0,l],L_\zeta^*(\mathbb{R}^3)\big)$ is a solution to \eqref{localf} with Grad's angular cut-off potential with $0<\gamma\leq 1$ and $f(0,\zeta)\in  L^ \infty_\zeta(\mathbb{R}^{3+})$  and $f(l,\zeta)\in  L^ \infty_\zeta(\mathbb{R}^{3-}) $.

  Then,
\begin{align}\label{lmGrad1}
& | \int_{\zeta_1>0}\phi_\alpha\frac{1}{|\zeta_1|}e^{-\frac{\nu(\zeta)}{|\zeta_1|}x}L(f)(0,\zeta)d\zeta| \leq C(| \ln x|+1)A_\alpha\langle f \rangle,\\&\label{lmGrad2}
| \int_{\zeta_1>0}\phi_\alpha\frac{1}{|\zeta_1|}e^{-\frac{\nu(\zeta)}{|\zeta_1|}x}(K(f)(x,\zeta)-K(f)(0,\zeta))d\zeta|\leq  CA_\alpha\langle f \rangle,\\&| \int_{\zeta_1>0}\phi_\alpha\int_0^x\frac{\nu(\zeta)}{|\zeta_1|^2}e^{-\frac{\nu(\zeta)}{|\zeta_1|}(x-s)}\big(K(f)(x,\zeta)-K(f)(s,\zeta)\big)dsd\zeta|\leq CA_\alpha\langle f \rangle \label{lmGrad3}
, \end{align}

where
\begin{equation}
\langle f \rangle:=|||f|||+\Vert f(0,\zeta)\Vert_{L^ \infty_\zeta(\mathbb{R}^{3+})}+\Vert f(l,\zeta)\Vert_{L^ \infty_\zeta(\mathbb{R}^{3-})}.
\end{equation}\end{lemma} 

\begin{proof}

We observe that $f$ is in fact bounded for all $x$ and $\zeta$ if $f(0,\zeta)\in  L^ \infty_\zeta(\mathbb{R}^{3+})$  and $f(l,\zeta)\in  L^ \infty_\zeta(\mathbb{R}^{3-}) $. For  $\zeta_1>0$,
\begin{equation}\begin{split} |f(x,\zeta)|&\leq |f(0,\zeta)|+C\Vert f\Vert_*\int_0^x\frac1{|\zeta_1|}e^{-\frac{\nu(\zeta)}{|\zeta_1|}(x-s)}ds\\&\leq \Vert f(0,\zeta)\Vert_{L^ \infty_\zeta(\mathbb{R}^{3+})}+C\Vert f\Vert_*\nu_0^{-1}\int_0^{\frac{x}{|\zeta_1|}}e^{-z}dz\\&\leq C( \Vert f(0,\zeta)\Vert_{L^ \infty_\zeta(\mathbb{R}^{3+})}+\Vert f\Vert_*).\end{split}\end{equation}
 A similar  inequality holds for $\zeta_1<0$. Therefore, 
\begin{equation} \Vert f \Vert_{L^\infty_{x,\zeta}} \leq C\langle f \rangle.\end{equation}

 We observe
 \begin{equation}\begin{split}
 &| \int_{\zeta_1>0}\phi_\alpha\frac{1}{|\zeta_1|}e^{-\frac{\nu(\zeta)}{|\zeta_1|}x}L(f)(0,\zeta)d\zeta| \\&\leq CA_\alpha\langle f \rangle|\int_{\zeta_1>0}e^{-\frac{|\zeta|^2}4}\frac{1}{|\zeta_1|}e^{-\frac{\nu(\zeta)}{|\zeta_1|}x}d\zeta|\leq CA_\alpha\langle f \rangle(|\ln x|+1),\end{split}
 \end{equation}
 which concludes \eqref{lmGrad1}

We will present the proof for \eqref{lmGrad3}.  

Replacing $0$, $l$ by $s$, $x$ in the \eqref{Integraleq1}, we can derive
\begin{equation}\begin{split}\label{Kexpres}
&K(f)(x,\zeta)-K(f)(s,\zeta)=\int_{\zeta_{*1}>0}k(\zeta,\zeta_*)(e^{-\frac{\nu(\zeta_*)}{|\zeta_{*1}|}|x-s|}-1)f(s,\zeta_*)d\zeta_*\\&+\int_{\zeta_{*1}>0}k(\zeta,\zeta_*)\int^x_{s}\frac1{|\zeta_{*1}|}e^{-\frac{\nu(\zeta_*)}{|\zeta_{*1}|}|x-t|}K(f)(t,\zeta_*)dtd\zeta_*\\&+\int_{\zeta_{*1}<0}k(\zeta,\zeta_*)(1-e^{-\frac{\nu(\zeta_*)}{|\zeta_{*1}|}|x-s|})f(x,\zeta_*)d\zeta_*\\&+\int_{\zeta_{*1}<0}k(\zeta,\zeta_*)\int^x_{s}\frac1{|\zeta_{*1}|}e^{-\frac{\nu(\zeta_*)}{|\zeta_{*1}|}|s-t|}K(f)(t,\zeta_*)dtd\zeta_*\\&=: H_1+H_2+H_3+H_4.
\end{split}\end{equation}
The term $H_2$ and $H_4$ have the property to be proved later
\begin{equation}\label{H2Holder}
|H_2|\leq C||f||_*|x-s|^{\beta}, \ |H_4|\leq C||f||_*|x-s|^{\beta},
\end{equation}
where $0<\beta<\frac{\gamma}{2+\gamma}$. 
We let \begin{equation}\begin{split} &\int_{\zeta_1>0}\int_0^x\phi_\alpha(\zeta)\frac{\nu(\zeta)}{|\zeta_1|^2}e^{-\frac{\nu(\zeta)}{|\zeta_1|}(x-s)}\big(K(f)(x,\zeta)-K(f)(s,\zeta)\big)dsd\zeta\\&=\int_{\zeta_1>0}\int_0^x\phi_\alpha(\zeta)\frac{\nu(\zeta)}{|\zeta_1|^2}e^{-\frac{\nu(\zeta)}{|\zeta_1|}(x-s)}\big(H_1+H_2+H_3+H_4)dsd\zeta \\&=:B_1+B_2+B_3+B_4.\end{split}\end{equation}Therefore, for $i=2$ and $4,$

\begin{equation}\begin{split}
&|B_i|\leq CA_{\alpha}||f||_*|\int_{\zeta_1>0}e^{-\frac{|\zeta|^2}4}\int_0^x\frac{\nu(\zeta)}{|\zeta_1|^2}e^{-\frac{\nu(\zeta)}{|\zeta_1|}(x-s)}|x-s|^\beta
 ds d\zeta|\\&\leq CA_\alpha ||f||_*|\int_{\zeta_1>0}e^{-\frac{|\zeta|^2}4}{\nu(\zeta)^{-\beta}}|\zeta_1|^{\beta-1}\int_0^{\frac{\nu(\zeta)x}{|\zeta_1|}}z^{\beta}e^{-z}
 dz d\zeta|\\&\leq CA_\alpha ||f||_*,\end{split}\end{equation}
where $z=\frac{\nu(\zeta)}{|\zeta_1|}(x-s)$. 
Estimates for $B_1$ and $B_3$ are not so obvious and the analysis is more demanding. We shall present the case for $B_1$ only and the case for $B_3$ can be done similarly. We claim
\begin{equation}\begin{split}
&|B_1|=\\&|\int_{\zeta_1>0}\phi_\alpha\int_0^x\frac{\nu(\zeta)}{|\zeta_1|^2}e^{-\frac{\nu(\zeta)}{|\zeta_1|}(x-s)}
 \int_{\zeta_{*1}>0}k(\zeta,\zeta_*)(e^{-\frac{\nu(\zeta_*)}{|\zeta_{*1}|}|x-s|}-1)f(s,\zeta_*)d\zeta_*ds d\zeta|\\&\leq CA_\alpha \langle f\rangle.\end{split}\end{equation}
 
Change the order of integration, we have

\begin{equation}\begin{split}
&|B_1|=\\&|\int_0^x\int_{\zeta_{*1}>0}\Big(\int_{\zeta_1>0}k(\zeta,\zeta_*)\phi_\alpha\frac{\nu(\zeta)}{|\zeta_1|^2}e^{-\frac{\nu(\zeta)}{|\zeta_1|}(x-s)}
 d\zeta\Big)(e^{-\frac{\nu(\zeta_*)}{|\zeta_{*1}|}|x-s|}-1)f(s,\zeta_*)d\zeta_*ds|.\end{split}\end{equation}
 
 We observe, for $a\geq0$
 \begin{align}
&\Vert\phi_\alpha\frac{\nu(\zeta)}{|\zeta_1|^2}e^{-\frac{\nu(\zeta)}{|\zeta_1|}(x-s)}\Vert_{L_a^1} \leq C_aA_\alpha |x-s|^{-1}(1+\big|\ln|x|\big|),\\&\Vert\phi_\alpha\frac{\nu(\zeta)}{|\zeta_1|^2}e^{-\frac{\nu(\zeta)}{|\zeta_1|}(x-s)}\Vert_{L_a^\infty} \leq C_aA_\alpha |x-s|^{-2}\label{phiainf}.\end{align}
Interpolating the inequalities above, we obtain
\begin{equation}\begin{split}\label{interf}
&\Vert\phi_\alpha\frac{\nu(\zeta)}{|\zeta_1|^2}e^{-\frac{\nu(\zeta)}{|\zeta_1|}(x-s)}\Vert_{L_a^p} \leq C_aA_\alpha |x-s|^{-2+\frac1p}(1+\big|\ln|x|\big|)^{\frac1p},\end{split}\end{equation}
where $1\leq p \leq \infty.$ In particular,
\begin{equation}\label{L2phialpha}
\Vert\phi_\alpha\frac{\nu(\zeta)}{|\zeta_1|^2}e^{-\frac{\nu(\zeta)}{|\zeta_1|}(x-s)}\Vert_{L^2} \leq C A_{\alpha} |x-s|^{-\frac32}(1+\big|\ln|x|\big|)^{\frac12}.
\end{equation}

Let $h(\theta,\gamma,a)=(\frac32-\gamma)\theta+(2+a-\gamma)(1-\theta)=2+a(1-\theta)-\gamma-\frac12\theta$, where $0\leq\theta\leq1$. We have \begin{equation} \begin{split}\Vert K(f)\Vert_{L^\infty_h}=&\sup_{\zeta\in\mathbb{R}^3}\left(|f(\zeta)|(1+|\zeta|)^{\frac32-\gamma}\right)^\theta\left(|f(\zeta)|(1+|\zeta|)^{2+a-\gamma}\right)^{1-\theta}\\& \leq \Vert K(f)\Vert_{L^\infty_{\frac32-\gamma}}^\theta\Vert K(f)\Vert_{L^\infty_{2+a-\gamma}}^{1-\theta}\leq C\Vert f\Vert_{L^2}^\theta\Vert f\Vert_{L^\infty_a}^{1-\theta}.\label{interpoK}\end{split}
\end{equation} 

 Combining \eqref{L2Linfty}, \eqref{Linfty}, \eqref{phiainf} ,  \eqref{L2phialpha}, and \eqref{interpoK}, we have
 \begin{equation} \Vert\int_{\zeta_1>0}\phi_\alpha\frac{\nu(\zeta)}{|\zeta_1|^2}e^{-\frac{\nu(\zeta)}{|\zeta_1|}(x-s)}d\zeta_1\Vert_{L^\infty_h}\leq C C_a^{1-\theta}A_\alpha|x-s|^{-2+\frac12\theta}(1+\big|\ln|x-s|\big|)^{\frac12\theta},\label{KfLinfh}
\end{equation} Applying \eqref{KfLinfh} above with fixed $0<\theta<1$ and  $a$ large enough,
we have
\begin{equation}\begin{split}\label{B1est}
&{|B_1|}\leq|\int_0^x\int_{\zeta_{*1}>0}(1+|\zeta_*|)^{-h(\theta,\gamma,a)}|x-s|^{-(2-\frac12\theta)}(1+\big|\ln|x-s|\big|)^{\frac12\theta}|\\&(e^{-\frac{\nu(\zeta_*)}{|\zeta_{*1}|}|x-s|}-1)|f(s,\zeta_*)|d\zeta_*ds|\\&\leq{ C A_\alpha} 
|\int_0^x|x-s|^{-(2-\frac12\theta)}(1+\big|\ln|x-s|\big|)^{\frac12\theta}\\&\ \ \ \ \int_0^{|x-s|}\int_{\zeta_{*1}>0}\frac{\nu(\zeta_*)}{|\zeta_{*1}|}e^{-\frac{\nu(\zeta_*)}{|\zeta_{*1}|}t}(1+|\zeta_*|)^{-h(\theta,\gamma,a)}|f(s,\zeta_*)|d\zeta_*dtds|\\&\leq C A_\alpha\langle f\rangle
|\int_0^x|x-s|^{-(1-\frac12\theta)}(1+\big|\ln|x-s|\big|)^{1+\frac12\theta}ds| \leq CA_\alpha\langle f\rangle.\end{split}\end{equation}

 The proof for \eqref{lmGrad2} is similar and simpler. 
  Replacing $s$ in \eqref{Kexpres} by 0 and denoting these terms as $H'_i$s, we write

\begin{equation} \begin{split}&\int_{\zeta_1>0} \phi_\alpha\frac{1}{\zeta_1}e^{-\frac{\nu(\zeta)}{|\zeta_1|}x}\left(K(f)(x,\zeta)-K(f)(0,\zeta)\right)d\zeta\\&=\int_{\zeta_1>0} \phi_\alpha\frac{1}{\zeta_1}e^{-\frac{\nu(\zeta)}{|\zeta_1|}x}(H_1'+H_2'+H_3'+H_4')d\zeta\\&=:B_1'+B_2'+B_3'+B_4'.\end{split}\end{equation}  

Using the fact \begin{equation} |H_2'|\leq C|x|^\beta\Vert f\Vert_* ,\ \ |H_4'|\leq C|x|^\beta\Vert f\Vert_*,\end{equation} 
we have
\begin{equation}\begin{split}|B_2'+B_4'|&\leq \Vert f\Vert_*\int_{\zeta_1>0} \phi_\alpha\frac{1}{\zeta_1}e^{-\frac{\nu(\zeta)}{|\zeta_1|}x}|x|^\beta d\zeta \\&\leq C\Vert f\Vert_*\int_{\zeta_1>0} \phi_\alpha\frac{1}{\nu(\zeta)|\zeta_1|^{1-\beta}} d\zeta\leq C\Vert f\Vert_*.\end{split} \end{equation}

The treatment for $B_1'$ and $B_3'$ is similar, and therefore we present the case for $B_1'$ only. Changing the order of integration, we have
 
\begin{equation} \begin{split}&B_1'=\int_{\zeta_1>0} \phi_\alpha\frac{1}{\zeta_1}e^{-\frac{\nu(\zeta)}{|\zeta_1|}x}\int_{\zeta_{*1}>0}k(\zeta,\zeta_*)(e^{-\frac{\nu(\zeta_*)}{|\zeta_{*1}|}x}-1)f(0,\zeta_*)d\zeta_*d\zeta\\&=\int_{\zeta_{*1}>0} (e^{-\frac{\nu(\zeta_*)}{|\zeta_{*1}|}x}-1)f(0,\zeta_*)\left(\int_{\zeta_{1}>0}k(\zeta,\zeta_*)\phi_\alpha\frac{1}{\zeta_1}e^{-\frac{\nu(\zeta)}{|\zeta_1|}x}d\zeta_*\right)d\zeta.\end{split}\end{equation}  

Note that\begin{align}
&\Vert\phi_\alpha\frac{1}{\zeta_1}e^{-\frac{\nu(\zeta)}{|\zeta_1|}x}\Vert_{L^1}\leq CA_\alpha(1+\big|\ln x\big|), \\
&\Vert\phi_\alpha\frac{1}{\zeta_1}e^{-\frac{\nu(\zeta)}{|\zeta_1|}x}\Vert_{L^\infty_a}\leq C_aA_\alpha(\frac1x). \label{LinfBprime}
\end{align}
By interpolation, we have
\begin{equation}
\Vert\phi_\alpha\frac{1}{\zeta_1}e^{-\frac{\nu(\zeta)}{|\zeta_1|}x}\Vert_{L^2} \leq CA_\alpha(\frac1x)^{\frac12}(1+\big|\ln x\big|)^{\frac12}\label{L2Bprime}.
\end{equation}
Combining \eqref{L2Linfty}, \eqref{Linfty}, \eqref{LinfBprime}, \eqref{L2Bprime}, and \eqref{interpoK}, we have
\begin{equation}\Vert\int_{\zeta_1>0} k(\zeta,\zeta_*)\phi_\alpha\frac{1}{\zeta_1}e^{-\frac{\nu(\zeta)}{|\zeta_1|}x}d\zeta\Vert_{L^\infty_h}\leq
CC_a^{1-\theta}A_\alpha (\frac1x)^{1-\frac{\theta}2}(1+\big|\ln x\big|)^{\frac12\theta}. \end{equation}
Similar to \eqref{B1est}, for fixed $0<\theta<1$ and $a$ large enough, we have

\begin{equation} \begin{split}&|B_1'|\leq CA_\alpha\int_{\zeta_{*1}>0} (e^{-\frac{\nu(\zeta_*)}{|\zeta_{*1}|}x}-1)f(0,\zeta_*)(1+|\zeta_*|)^{-h(\theta,\gamma,a)}(\frac1x)^{1-\frac{\theta}2}(1+\big|\ln x\big|)^{\frac12\theta}d\zeta_*\\&\leq CA_\alpha (\frac1x)^{1-\frac{\theta}2}(1+\big|\ln x\big|)^{\frac12\theta}\int_0^x\int_{\zeta_{*1}>0}\frac{\nu(\zeta_*)}{|\zeta_{*1}|}e^{-\frac{\nu(\zeta_*)}{|\zeta_{*1}|}t}(1+|\zeta_*|)^{-h(\theta,\gamma,a)}|f(0,\zeta_*)|d\zeta_*dt\\&\leq CA_\alpha \langle f\rangle x^{\frac{\theta}2}(1+\big|\ln x\big|)^{1+\frac12\theta}\leq CA_\alpha \langle f\rangle .\end{split}\end{equation}

We still have to prove \eqref{H2Holder}.  We will present the proof for $H_2$ only and the one for $H_4$ is similar. We will first present the following lemma 
\begin{lemma}If  $f\in L^2_*$ and $\theta\in(\frac{2}{2+\gamma},1)$, then \label{lemmatheta}
\begin{equation}
\int\frac1{|\zeta_{*1}|^{2-2\theta}\nu(\zeta_*)^{2\theta}}|K(f)|^2d\zeta_*\leq C\Vert f\Vert_*^2.
\end{equation}
\end{lemma}

The proof Lemma \ref{lemmatheta} follows the idea of the one of  Lemma  4.2  in \cite{GP}. We  will present the proof at the end of this section to make this paper self-contained. 

With the Lemma \ref{lemmatheta} above, we have
 \begin{equation}\begin{split}|H_2|&=|\int^x_{s}\int_{\zeta_{*1}>0}k(\zeta,\zeta_*)\frac1{|\zeta_{*1}|}e^{-\frac{\nu(\zeta_*)}{|\zeta_{*1}|}|x-t|}K(f)(t,\zeta_*)d\zeta_*dt|\\&\leq  C(\int_s^x(\int_{\zeta_{*1}>0}|\frac1{|\zeta_{*1}|}e^{-\frac{\nu(\zeta_*)}{|\zeta_{*1}|}|x-t|}K(f)(t,\zeta_*)|^2d\zeta_*)^{\frac12}dt\\&
 \leq    C(\int_s^x\frac1{|x-t|^\theta}(\int_{\zeta_{*1}>0}|\frac1{|\zeta_{*1}|^{2-2\theta}\nu(\zeta_*)^{2\theta}}K(f)(t,\zeta_*)|^2d\zeta_*)^{\frac12}dt\\&\leq C|x-s|^{1-\theta}=C|x-s|^{\beta}. \end{split}\end{equation} \end{proof}
 
For $\zeta_1<0$, we can yield a similar lemma.  Together, we have

  \begin{lemma}\label{sigmapm}Suppose $f\in L^\infty_x\big([0,l],L_\zeta^*(\mathbb{R}^3)\big)$ is a solution to \eqref{localf} with Grad's angular cut-off potential with $0<\gamma\leq 1$ and $f(0,\zeta)\in  L^ \infty_\zeta(\mathbb{R}^{3+})$  and $f(l,\zeta)\in  L^ \infty_\zeta(\mathbb{R}^{3-}) $.
 Then,
  \begin{align}|\partial_x\sigma^+_\alpha(x)|&\leq C(\big|\ln|x|\big|+1),\label{sigmap}\\|\partial_x\sigma^-_\alpha(x)|&\leq C(\big|\ln|l-x|\big|+1),\label{sigmam}
\end{align}  where\begin{equation}\sigma^+_\alpha(x)=\int_{\zeta_1>0}\phi_\alpha f d\zeta,\ \ \sigma^-_\alpha(x)=\int_{\zeta_1<0}\phi_\alpha f d\zeta.\end{equation}\end{lemma}

 \begin{proof}[Proof for Lemma \ref{lemmatheta}]
 We observe
 \begin{equation}\begin{split}&\int|Kf|^2d\zeta=\int\left(\int k(\zeta,\zeta_*)f(\zeta*)d\zeta*\int k(\zeta,\zeta')f(\zeta')d\zeta'\right)d\zeta\\&\leq C\Vert f\Vert_{L^2}\int | f(\zeta_*)|\int k(\zeta,\zeta_*)(1+|\zeta|)^{-(\frac32-\gamma)} d\zeta d\zeta_*\\ &\leq C\Vert f\Vert_{L^2}\int| f(\zeta_*)|(1+|\zeta*|)^{-(\frac72-2\gamma)}d\zeta_*\leq C\Vert f\Vert_{L^2}\Vert f\Vert_{*}\leq C\Vert f\Vert_{*}^2.\end{split}\end{equation}
Together with \eqref{L2Linfty},  we know $|Kf|^2\in L^\infty\cap L^1$. Interpolating between these two inequalities, we have
\begin{equation}\Vert |Kf|^2 \Vert_{L^p}\leq C \Vert f\Vert_*^2 \mbox{ for }  1 \leq p\leq \infty .\end{equation}
Therefore, we now only need to prove for  some proper $0<\theta<1$ and  H{\"o}lder conjugate of $p$, $p'\in [1,\infty]$, 
\begin{equation}
\int\frac1{|\zeta_{*1}|^{(2-2\theta)p'}|1+|\zeta_*|)|^{2\gamma\theta p'}}d\zeta_*<\infty,
\end{equation}
which yields the following condition:

\begin{equation}
(2-2\theta)p'<1;\ (2-2\theta+2\gamma\theta)p'>3.
\end{equation}
Such $p'$ exists if and only if
\begin{equation}
\theta<1,\ 3(2-2\theta)<2-2\theta+2\gamma\theta,
\end{equation}
which concludes Lemma \ref{lemmatheta}.
\end{proof}

\section{Asymptotic formula\label{section3}}In the precious section, we obtain an upper bound for $|\partial_x\sigma_\alpha|$, which diverges to infinity at boundary like a logarithmic function. Through the analysis, we also localize  the source of singularity,
which is the contribution from the first term on the right hand side in \eqref{dfdxf}. In this section, restricted   to the inverse-power potential with angular cut-off,  the goal is to further single out and factorize the singularity and form an asymptotic formula, i.e.,
\begin{lemma} \label{DsigmaS}Suppose $f(0,\zeta)\in L_*$ and $\nabla f(0,\zeta)\in L^p_\zeta(\mathbb{R}^{3+})$. Then, \begin{equation} \begin{split}\frac{\partial}{\partial x}\sigma_{\alpha1}^+:=&\int_{\zeta_1>0}\phi_{\alpha}
\frac{1}{|\zeta_1|}e^{-\frac{\nu(|\zeta|)}{|\zeta_1|}x}L(f)(0,\zeta)d\zeta\\&=-\ln x\int\int \phi_\alpha(0,\zeta_2,\zeta_3)L(f)(0,0^+,\zeta_2,\zeta_3)d\zeta_2d\zeta_3+O(A_\alpha \langle g \rangle'),\end{split}
\end{equation}  where
\begin{equation}
L(f)(0,0^+,\zeta_2,\zeta_3):=\lim_{\zeta_1\to 0^+}L(f)(0,\zeta_1,\zeta_2,\zeta_3).\end{equation}

\end{lemma}

   \begin{proof}
  
  If we change to spherical coordinates so that \[\zeta=(\rho\cos\theta, \rho\sin\theta\cos\phi,\rho\sin\theta\sin\phi),\] we have
  
  \begin{equation}\frac{\partial}{\partial x}\sigma_{\alpha1}^+=\int_0^\infty\int_0^{2\pi}\int_0^{\frac{\pi}{2}}\label{sigmapsh}
\frac{1}{\rho\cos\theta}e^{-\frac{\rho^2}2}e^{-\frac{\nu(\rho)}{\rho\cos\theta}x}\bar{F}(\rho,\theta,\phi)\rho^2\sin\theta d\theta d\phi d\rho,\end{equation}
  where 
  \begin{equation}\bar{F}(\rho,\theta,\phi)=\pi^{-\frac34}\zeta_1^{\alpha_1}\zeta_2^{\alpha_2}\zeta_3^{\alpha_3}L(f)(0,\zeta)=:g(\zeta).\end{equation}
  Further letting  $z=\cos\theta$, we obtain 
   \begin{equation}\frac{\partial}{\partial x}\sigma_{\alpha1}^+=\int_0^\infty\int_0^{2\pi}\left[\int_0^{1}\label{sigmapshz}
\frac{1}{z}e^{-\frac{\nu(\rho)}{\rho z}x}{F}(\rho, z,\phi) d z \right]e^{-\frac{\rho^2}2}\rho d\phi d\rho,\end{equation}
  where 
  \begin{equation}F(\rho,z ,\phi)=\bar{F}(\rho,\cos^{-1}z,\phi).\end{equation}
  
  Here, we introduce a well-know special function, exponential integral,
  \begin{equation} E_1(x)=\int_0^1\frac1{z}e^{-\frac{x}{z}}dz.\end{equation}
  The $E_1(x)$ has the following  properties, \cite{handbook}:
\begin{equation}
E_{1}(x)=-\gamma-\ln x+\sum_{k=1}^{\infty}\frac{(-1)^{k+1}x^{k}}{k\cdot k!},\label{expandion}
\end{equation}

\begin{equation}
\frac{1}{2}e^{-x}\ln(1+\frac{2}{x})\leq E_{1}(x)\leq e^{-x}\ln(1+\frac{1}{x})\ \ \mbox{for}\ x>0.\label{e1far}
\end{equation}
From the properties above, we have 
\begin{equation}
E_{1}(x)=-\ln x+O(1)\label{E1o1}\  \mbox{, for } 0<x\leq1,
\end{equation}

  Let \begin{equation} H(z,x)=-\int_z^1\frac{1}{u}e^{-\frac{\nu(\rho)}{\rho u}x}du.\end{equation}
  Notice that  $\frac{\partial}{\partial z} H(z,x)=\frac{1}{z}e^{-\frac{\nu(\rho)}{\rho z}x}$and $H(0,x)=E_1(\frac{\nu(\rho)}{\rho }x)$. 
  Performing integration by parts for the inner most integral in \eqref{sigmapshz}, we obtain
  \begin{equation}\begin{split}&\int_0^1\left(\frac{\partial}{\partial z}H(z)\right)F(\rho,z,\phi)dz\\&=E_1(\frac{\nu(\rho)}{\rho }x)F(\rho,0,\phi)-\int_0^1H(z)\left(\frac{\partial}{\partial z}F(\rho,z,\phi)\right)dz\\.\end{split}\end{equation}
  The first  term on the right hand side above is the source of singularity and will be explained in detail later. We will prove the contribution from the second term above is finite.  
  Let \begin{align}I:&=\int_0^\infty\int_0^{2\pi}E_1(\frac{\nu(\rho)}{\rho}x)F(\rho,0,\phi)e^{-\frac{\rho^2}2}\rho d\phi d\rho.\label{Isigular},\\II:&=\int_0^\infty\int_0^{2\pi}\int_0^1H(z)\left(\frac{\partial}{\partial z}F(\rho,z,\phi)\right)dz e^{-\frac{\rho^2}2}\rho d\phi d \rho.&\end{align}
  
  Notice that 
  \begin{equation} |H(z,x)|\leq |\ln z|,\end{equation} 
  
  \begin{equation}\frac{\partial}{\partial z} F(\rho, z,\phi)=\rho\frac{\partial}{\partial \zeta_1}g+\frac{\rho z}{\sqrt{1-z^2}}\cos\phi \frac{\partial}{\partial \zeta_2}g+\frac{\rho z}{\sqrt{1-z^2}}\sin\phi\frac{\partial}{\partial \zeta_3}g.\end{equation} 
  
  We have \begin{equation} \begin{split}\left|II\right| &\leq \left|\int_0^\infty\int_0^{2\pi}\int_0^1H(z)\left(\frac{\partial}{\partial z}F(\rho,z,\phi)\right)dz e^{-\frac{\rho^2}2}\rho d\phi d \rho\right| \\&\leq \left|\int_0^\infty\int_0^{2\pi}\left(\int_0^1C(|\frac{\partial}{\partial \zeta_3}g|+|\frac{\partial}{\partial \zeta_2}g|)dz +\int_0^1|\ln z||\frac{\partial}{\partial \zeta_1}g| dz \right)e^{-\frac{\rho^2}2}\rho^2 d\phi d \rho\right|\\ &\leq CA_\alpha \int_{\zeta_1>0}e^{-\frac{|\zeta|^2}8}\left(|L(0,\zeta)|+|\frac{\partial}{\partial \zeta_2}L(0,\zeta)|+|\frac{\partial}{\partial \zeta_3}L(0,\zeta)|\right) d\zeta\\&+A_\alpha\left|\int_0^\infty\int_0^{2\pi}\left(\int_0^1|\ln z|^q dz \right)^{\frac1q}\left(\int_0^1|\frac{\partial}{\partial \zeta_1}L(0,\zeta)|^p dz \right)^{\frac1p}e^{-\frac{\rho^2}4}\rho^2 d\phi d \rho\right|,\end{split}\end{equation}
  
  where $p>1$ and $\frac1p+\frac1q=1$.
  
  Let the second term on the right hand side above be $II'$.
  \begin{equation}\begin{split} &|II'| \leq \left(\int_0^\infty\int_0^{2\pi}\left(\int_0^1|\ln z|^q dz \right)e^{\frac{q\rho^2}8}\rho^2 d\phi d \rho\right)^{\frac1q}\left(\int_0^\infty\int_0^{2\pi}\left(\int_0^1|\frac{\partial}{\partial \zeta_1}L(0,\zeta)|^p dz \right)e^{\frac{p\rho^2}8}\rho^2 d\phi d \rho\right)^{\frac1p}\\&\leq C_p A_\alpha\left( \int_{\zeta_1>0}|\frac{\partial}{\partial \zeta_1}L(0,\zeta)|^p e^{-\frac{p|\zeta|^2}8} d\zeta\right)^{\frac1p}\end{split} .\end{equation}
  Therefore,\begin{equation}\begin{split}|II|&\leq CA_\alpha \int_{\zeta_1>0}e^{-\frac{|\zeta|^2}8}\left(|L(0,\zeta)|+|\frac{\partial}{\partial \zeta_2}L(0,\zeta)|+|\frac{\partial}{\partial \zeta_3}L(0,\zeta)|\right) d\zeta\\&+ C_p A_\alpha\left( \int_{\zeta_1>0}|\frac{\partial}{\partial \zeta_1}L(0,\zeta)|^p e^{-\frac{p|\zeta|^2}8} d\zeta\right)^{\frac1p},\end{split}\end{equation}
  where $p>1$.
  
  Using the assumption and \eqref{dKLp}, we have
  \begin{equation}|II|\leq C A_\alpha(\Vert\nabla f(0,\zeta)\Vert_{L^p_\zeta(\mathbb{R}^{3+})}+\Vert  f(0,\zeta)\Vert_*)\end{equation}
  
  Finally, we are going to extract the singularity from  $I$.
   We let \begin{equation}\rho_0=\sup\{\rho|\frac{\nu(\rho)}\rho x>1\}.\end{equation}
 We devided the domain of integration of $I$ in \eqref{Isigular} into two, $0\leq\rho\leq \rho_0$ and $\rho_0\leq \rho$, and denoted the integral as $I_s$ and $I_l$ correspondently.  
 Note that if $0\leq \rho<\rho_0$, then  \begin{equation}\frac{\nu(\rho)}{\rho}x> \frac{c_0}{c_1}.\end{equation} Applying \eqref{e1far}
, we have  \begin{equation}
 |I_s|\leq C A_\alpha \langle f \rangle.\end{equation}   Using the asymptotic formula \eqref{E1o1}, We obtain
   \begin{equation}\begin{split}I_l&=-\ln(x) \int_{\rho_0}^\infty\int_0^{2\pi}F(\rho,0,\phi)e^{-\frac{\rho^2}2}\rho d\phi d\rho+O(A_\alpha\langle f\rangle)\\&= -\ln x\int\int \phi_\alpha(0,\zeta_2,\zeta_3)L(f)(0,0^+,\zeta_2,\zeta_3)d\zeta_2d\zeta_3+O\left(A_\alpha \langle f \rangle(1+\rho_0^2|\ln x| )\right).\end{split}
   \end{equation}
   The remaining task is to estimate $\rho_0$. If we assume $x\leq\frac1{2c_1}$,  then\begin{equation}\label{nurho}
   1=\frac{\nu(\rho_0)}{\rho_0}x\leq c_1\frac{(1+\rho_0)^\gamma}{\rho_0}x\leq \frac{(1+\rho_0)^\gamma}{2\rho_0}.\end{equation}
   If $\gamma=1$, we see $\rho_0\leq 1.$
   Observe that, for $0\leq\gamma<0$,
   $ \rho$  grows faster then the 
  $\frac12(1+\rho)^\gamma$ as $\rho\to \infty$. Therefore, \eqref{nurho} implies $\rho_0\leq m$ for some $m<\infty$. 
  Therefore, we have
 \begin{equation}\label{nurho}
   1=\frac{\nu(\rho_0)}{\rho_0}x\leq c_1\frac{(1+m)^\gamma}{\rho_0}x.\end{equation}   
   We have
   \begin{equation}
   |\rho_0^2\ln x|\leq Cx^2|\ln x| \leq C
      \end{equation}
      and conclude the lemma.\end{proof}
      
      Finally, combining \eqref{lmGrad2} and \eqref{lmGrad3} in Lemma \ref{lmGrad}, \eqref{sigmam} in Lemma \ref{sigmapm}, and Lemma \ref{DsigmaS}, we concludes Theorem \ref{mainthm}.


\begin{thebibliography}{99}

\bibitem{handbook}  Abramowitz, M.; Stegun, I.: Handbook of Mathematical Functions: With Formulas, Graphs, and Mathematical Tables, Courier Dover Publications, 1964



\bibitem{AN1}
Arkeryd, L.; Nouri, A.: A compactness result related to the stationary Boltzmann equation in a slab, with applications to the existence theory. Indiana Univ. Math. J. 44 (1995), no. 3, 815Ð839. 
\bibitem {AN}Arkeryd, L.; Nouri, A.: The stationary Boltzmann equation in the slab with given weighted mass for hard and soft forces. Ann. Scuola Norm. Sup. Pisa Cl. Sci. (4) 27 (1998), no. 3-4, 533Ð556 (1999).
\bibitem{BCN}Bardos, C.; Caflish, E.; Nicolaenko, B.: The Milne and Kramers problems for the
Boltzmann equation of a hard sphere gas, Comm. Pure Appl. Math. 
(1986).
\bibitem{BGS}Bardos, C.; Golse, F.; Sone, Y.: Half-space problems for the Boltzmann equation: a survey. J. Stat. Phys. 124 (2006), no. 2-4, 275Ð300.
\bibitem{Caflisch} Caflisch, R.: The Boltzmann equation with a soft potential. I. Linear, spatially-homogeneous. Comm. Math. Phys. 74 (1980), no. 1, 71Ð95.
\bibitem{CCLS} Chen, C.-C.; Chen, I-K.; Liu, T.-P.; Sone, Y.:
Thermal transpiration for the linearized Boltzmann equation.
\textit{Comm. Pure Appl. Math.} \textbf{60}, 147--163 (2007).

\bibitem{IKC} Chen, I-K.  Boundary singularity  of  moments for the  linearized Boltzmann equation  \textit{Journal of statistical Physics} 
\bibitem{CFLT} Chen, I-K.; Funagane, H.; Takata, S.; Liu, T-P.: Singularity of the velocity distribution function in molecular velocity space 2013 (preprint)

\bibitem{CLT} Chen, I-K.; Liu, T-P.; Takata, S.:  Boundary singularity for thermal transpiration problem of the linearized Boltzmann equation
\bibitem{GP} Golse, F.; Poupaud, F.: Stationary solutions of the linearized Boltzmann equation in a half-space. Math. Methods Appl. Sci. 11 (1989), no. 4, 483Ð502. 
\bibitem{GPS} Golse, F.; Perthame, B.; Sulem, C.: On a boundary layer problem for the nonlinear Boltzmann equation. Arch. Rational Mech. Anal. 103(1), 81Ð96 (1988)
\bibitem{Grad} Grad, H: Asymptotic Theory of the Boltzmann Equation, II
  \bibitem {LY1}Liu, T.-P.; Yu, S.-H.: Solving  the Boltzmann equation, Part I: Green's function,  \textit{Bull. Inst. Math. Acad. Sin. (N.S.)} \textbf{6} (2006), 115-243.

\bibitem {LY2}Liu, T.-P.; Yu, S.-H.: Invariant manifolds for steady Boltzmann flows and applications, \textit{Archive for Rational Mechanics and Analysis} (to appear).

\bibitem{ohwadaSA} Ohwada, T.; Sone, Y.; Aoki, K.:
Numerical analysis of the Poiseuille and thermal transpiration flows between two parallel plates on the basis of the Boltzmann equation for hard-sphere molecules.
\textit{Phys. Fluids A} \textbf{1}, 2042--2049 (1989).

\bibitem{Sone64} Sone, Y.: Kinetic theory analysis of linearized Rayleigh problem. J. Phys. Soc. Jpn 19,
1463Ð1473.(1964)

\bibitem{Sone65}Sone, Y.: Effect of sudden change of wall temperature in a rarefied gas. J. Phys. Soc. Jpn
20, 222Ð229.(1965)


\bibitem{sone} Sone, Y.:
\textit{Molecular Gas Dynamics. Theory, Techniques, and Applications.}
Birkh\"{a}user, Boston, 2007.

\bibitem{SO78} Sone, Y.  Onishi, Y.: Kinetic theory of evaporation and condensation Ð Hydrodynamic
equation and slip boundary condition. J. Phys. Soc. Jpn 44, 1981Ð1994.(1978)


\bibitem{TakataFunagane} Takata, S.; Funagane, H.:
Poiseuille and thermal transpiration flows of a highly rarefied gas:
Over-concentration in the velocity distribution function.
\textit{J. Fluid Mech.} (accepted).



\bibitem{Takata}  Takata, S.; Funagane, H.:  Singular behaviour of a rarified gas on a planar boundary (priprint)  

\bibitem{UYY}Ukai, S.; Yang, T.; Yu, S-H.: Nonlinear boundary layers of the Boltzmann equation. I. Existence. Comm. Math. Phys. 236 (2003), no. 3, 373Ð393. 

\bibitem{UYY2} Ukai, S.; Yang, T.; Yu, S-H.: Nonlinear stability of boundary layers of the Boltzmann equation. I. The case $M_\infty<-1$. Comm. Math. Phys. 244 (2004), no. 1, 99Ð109.

\end{thebibliography}
\end{document}